\documentclass[11pt,a4paper]{article}
\usepackage{amsthm}
\usepackage{amsmath}
\usepackage{amssymb}
\usepackage{amsfonts}
\usepackage{xpatch}
\usepackage{mathtools}
\usepackage{enumitem}
\usepackage{tikz}

\newcommand{\defeq}{\vcentcolon=}

\usepackage[textfont=it,format=plain,justification=centering,margin={0.5cm,0.5cm}, font=small]{caption}
\makeatletter
\AtBeginDocument{\xpatchcmd{\@thm}{\thm@headpunct{.}}{\thm@headpunct{}}{}{}}
\makeatother

\theoremstyle{definition}
\newtheorem{thm}{Theorem}[section] 

\newtheorem*{thm*}{Theorem}
 
\newtheorem{defn}[thm]{Definition} 
 
\newtheorem{exmp}[thm]{Example}

\newtheorem{cor}[thm]{Corollary}
\theoremstyle{remark}
\newtheorem*{rem}{Remark:} 
\newcommand{\smf}{submanifold } 
\newcommand{\R}{\mathbb{R}}
\newcommand{\map}{S^{n+k} \rightarrow S^n}
\newcommand{\squeezeup}{\vspace{-3mm}}

\title{Transversality and framed cobordism}
\author{Sturmius Tuschmann}
\date{}
\begin{document}

\pagenumbering{gobble}

 \normalsize
\maketitle

\pagenumbering{arabic}
\begin{abstract}
\noindent René Thom's remarkable and far-reaching concept of transversality has found numerous powerful applications. Most importantly, it allowed Thom to develop cobordism theory, which led to a piercing insight into the topology of smooth manifolds, when he generalized Lev Pontryagin's earlier concept of framed cobordism. For his profound findings, in 1958 he was awarded the Fields Medal. \\ The present paper provides a down-to-earth approach to transversality theory which only assumes basic knowledge about smooth manifolds. Following an idea of Andrew Putman, we then use transversality to give a novel and detailed proof of an isomorphism theorem of Pontryagin about homotopy groups of spheres whose original proof is generally believed to be rather hard to access. %
\end{abstract}
\section*{Introduction}
Cobordism theory was established by Lev Pontryagin and above all by René Thom in the 1950s, but its roots already date back to the late 19th century. Namely, according to Jean Dieudonné \cite{Di89}, in 1895, during his attempts to define homology, Henri Poincaré consi\-dered systems of connected, closed 
$k$-dimensional smooth submanifolds $M_1,\dots,M_\lambda$ of a higher-dimensional smooth manifold $N$, whose union forms the boundary $\partial X$ of a connected, compact ($k$+1)-dimensional smooth submanifold $X$ of $N$. He expressed this relation by writing $M_1+\dots+M_\lambda \sim 0$ and called it a "homology" between the $M_j$. But when he claimed that one could "add" homologies, he ran into difficulties because he did not pay enough attention to the following question: Given two smooth submanifolds $X$ and $X'$ of $N$, how can one find a smooth submanifold $Y$ of $N$ with $\partial Y =\partial X\cup\partial X'$? This is indeed a nontrivial problem because $X\cup X'$ is not a smooth manifold in general, as $X\cap X'$ may not be empty.\\[1ex]
Those complications can be avoided if one considers the "sum" of smooth manifolds to be their disjoint union, like Pontryagin and Thom did. Their general idea was to consider smooth manifolds up to an equivalence relation called cobordism, i.e. two closed smooth $k$-manifolds $M$ and $M'$ are called cobordant, if their disjoint union $M\sqcup M'$ is equal to the boundary $\partial X$ of a compact smooth ($k$+1)-manifold $X$. \\[1ex]
The development of this theory started in 1950, when Pontryagin initially regarded a special case of cobordism, called framed cobordism, between closed smooth $k$-submanifolds of $\R^{n+k}$ with smoothly trivialized normal bundle, called framed manifolds. His primary motivation to do so was to investigate continuous maps between spheres. He showed that framed cobordism is an equivalence relation in every dimension and that the disjoint union operation turns the corresponding sets of equivalence classes into abelian groups $\Omega_k^{fr}(\R^{n+k})$. He was able to prove that these abelian groups are isomorphic to certain homotopy groups of spheres that he could determine thereby. \\[1ex]
In 1953, Thom broadly generalized and extended Pontryagin's work and hence can be seen as the main developer of cobordism theory. He came up with the great idea of transversality, which generalizes the concept of regu\-lar values of smooth maps between smooth manifolds. He also associated to an $n$-dimensional real vector bundle $\xi$ the quotient space of its disk bundle by its sphere bundle, called the Thom space $Th(\xi)$ of $\xi$. Then, using transversality, he discovered a correspondence between smooth submanifolds of a smooth manifold $N$ and continuous maps $N\to Th(\xi)$. From this, one can derive a bijection (even an isomorphism) between the $k$-th cobordism group $\mathcal{N}_k$ - obtained by equipping the set of cobordism classes of closed smooth $k$-manifolds with an addition induced by the disjoint union - and higher-dimensional homotopy groups of the Thom space of a special fiber bundle called universal bundle. By using the above isomorphism, Thom was able to determine the cobordism groups $\mathcal{N}_k$, i.e. classify closed smooth manifolds up to cobordism in his famous paper from 1954 \cite{Th54}. He was awarded the Field's medal particularly for this result four years later. \\[1ex] 
Thom's key idea was the simplification of a geometric cobordism problem to a homotopy problem. In contrast to Pontryagin, who solved a cobordism problem in order to compute certain homotopy groups, Thom made use of his one-to-one correspondence the other way around.\\[1ex]
The first two parts of this paper deal with Pontryagin's work on framed cobordism and Thom's concept of transversality, so they are arranged in the chronological order discussed above. However, instead of following up with Thom's work on cobordism, the last part of this paper focuses on the proof of Pontryagin's original theorem by using transversality methods as a special example for their application. This means we prove a result from 1950 by using the more "modern" tool of transversality established a few years later. This allows us to simplify certain arguments of the proof. 
\\[1ex]
\noindent I would like to thank Michael Joachim for his constant support and feedback he gave me during the writing of this paper which grew out of my thesis under his supervision.

\section{Pontryagin's work on framed cobordism}
After there had been a rapid development of algebraic topology in the 1940s, Pontryagin came up with a new approach to the homotopy classification of maps from the ($n$+$k$)-sphere to the $n$-sphere for some $k\geq 0$ in 1950. He regarded closed smooth $k$-manifolds embedded into $\R^{n+k}$ with smoothly trivialized normal bundle, so-called framed submanifolds of $\R^{n+k}$, and introduced an equivalence relation called framed cobordism between them. He showed that two smooth maps $f,g:\map$ are homotopic if and only if their pre\-images $f^{-1}(y)$, $g^{-1}(y)$ for a regular value $y\in S^n$ belong to the same framed cobordism class. Thereby, he simplified the problem of classifying maps $f:\map$ up to homotopy to the problem of classifying framed mani\-folds up to framed cobordism. By this, he was able to compute $\pi_{n+1}(S^n)$ and $\pi_{n+2}(S^n)$, followed by his student Vladimir Rokhlin, who later computed $\pi_{n+3}(S^n)$. \\[1ex]
If not stated otherwise, let $n\geq 1$ and $k\geq 0$.
\begin{defn}\label{normalbundle}
Let $M$ be a smooth $k$-\smf of  $\mathbb{R}^{n+k}$, where $\mathbb{R}^{n+k}$ is equipped with the standard scalar product. The normal space $N_xM\vcentcolon= T_xM^\perp \subset T_x\R^{n+k}$ of $M$ at $x$ is the space of normal vectors to $M$ in $\R^{n+k}$ at $x$. The normal bundle $NM$ of M in $\mathbb{R}^{n+k}$ is the smooth vector bundle over $M$ consisting of all pairs $(x,v)$ with $x \in M$ and $ v \in N_xM$. 
\end{defn}
\noindent The following theorem is a very useful tool in differential topology when working with normal bundles. A proof can be found in \cite{Le12}, Ch. 6.
\begin{thm}\label{tubularneighborhood} \textbf{(Tubular Neighborhood Theorem)}
Let $M$ be a smooth $k$-\smf of  $\mathbb{R}^{n+k}$. Then there is an open neighborhood of the zero section $(NM)_0\defeq\{(x,0)\in NM | x \in M\}$ in $NM$ which is mapped through the map \[\gamma:NM\to\R^{n+k}\text{, }(x,v)\mapsto x+v\] diffeomorphically onto an open neighborhood $U\subset\R^{n+k}$ of $M$. $U$ is called a tubular neighborhood of $M$.
\end{thm}
\begin{figure}[h]
\begin{center}
\squeezeup
\includegraphics[width = 250pt]{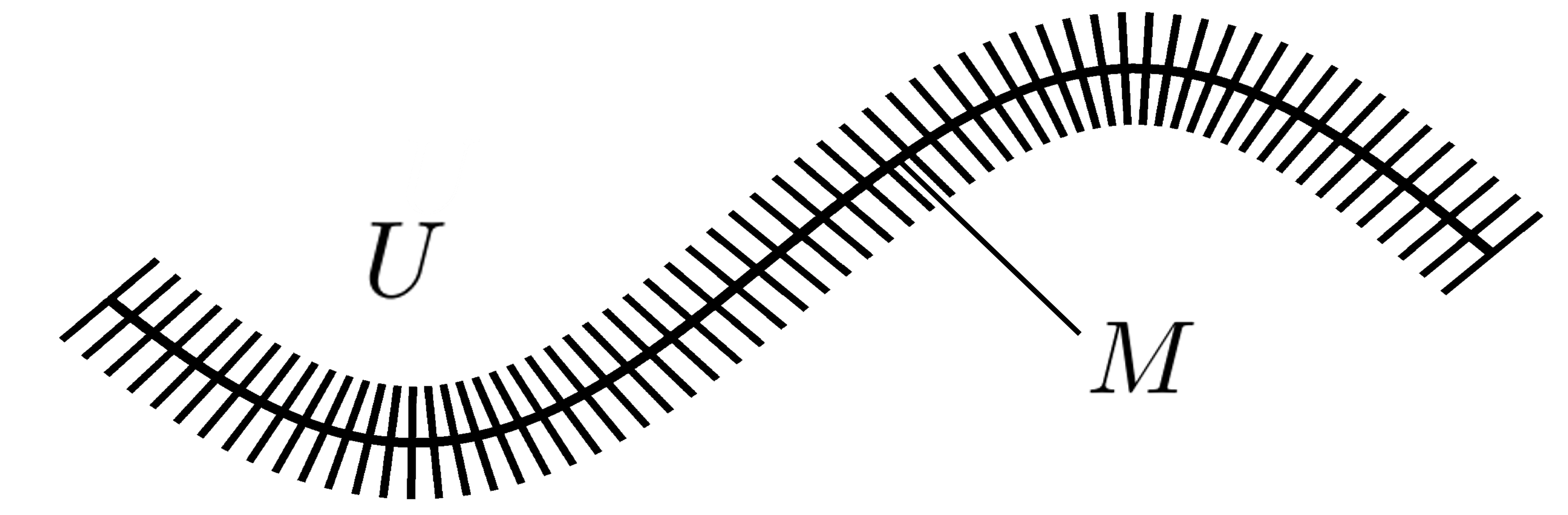}
\caption{A tubular neighborhood $U$ of $M$}
\end{center}
\squeezeup
\end{figure}
\begin{defn}
A framing of a smooth $k$-submanifold $M$ of $\mathbb{R}^{n+k}$ is a smooth map $\nu : M \rightarrow (NM)^n$ which assigns to each $x \in M$ a basis \[\nu(x) = (\nu_1(x), \dots, \nu_n(x))\] for $N_xM\subset T_x\R^{n+k}\cong \R^{n+k}$, and the $\nu_i(x)$ can be viewed as vectors in $\R^{n+k}$. If $M$ is closed, $(M, \nu)$ is called a framed submanifold of  $\mathbb{R}^{n+k}$.
\end{defn}
\begin{figure}[h]
\begin{center}

\includegraphics[width = 300pt]{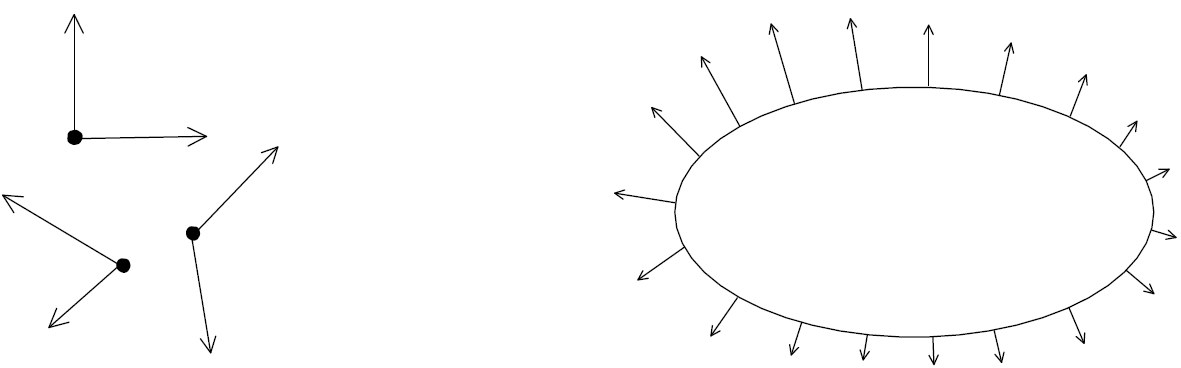}
\caption{\label{framedsmfs} A 0-dimensional and a 1-dimensional framed submanifold of $\R^2$}
\end{center}
\squeezeup
\end{figure}
\begin{rem}
Equivalently, a framing of $M$ is a smooth isomorphism of smooth vector bundles $M\times\mathbb{R}^n \cong NM$, which means that $M$ has a smoothly tri\-via\-lized normal bundle. Not every smooth submanifold of $ \mathbb{R}^{n+k}$ is frameable, for instance, no nonorientable smooth manifold can be given a framing.
\end{rem}
\begin{defn}\label{fc}
A framed cobordism between two $k$-dimensio\-nal framed submanifolds $(M, \nu)$ and $(M', \nu')$ of  $\mathbb{R}^{n+k}$ is a compact smooth ($k$+1)-submanifold $X$ of $\mathbb{R}^{n+k}\times[0,1]\subset\mathbb{R}^{n+k+1}$ such that \begin{itemize}
\item $(M\times[0,\varepsilon))\cup (M'\times(1-\varepsilon,1])\subset X$ for an $\varepsilon>0$
\item $\partial X = X\cap (\mathbb{R}^{n+k}\times\{0,1\}) = (M\times\{0\})\cup(M'\times\{1\})$
\end{itemize}
together with a framing $\omega$ on $X$ that restricts to the given framings $\nu$, $\nu'$ in the following way:
\[\squeezeup\omega_i(x,t) = (\nu_i(x),0) \text{ for } (x,t) \in M\times[0,\varepsilon) \text{ and}\] 
\[\omega_i(x,t) = (\nu'_i(x),0) \text{ for } (x,t) \in M'\times(1-\varepsilon,1]\] If such a framed cobordism exists, we call $M$ and $M'$ framed cobordant.
\end{defn}
\begin{figure}[h]
\begin{center}
\includegraphics[width = 170pt]{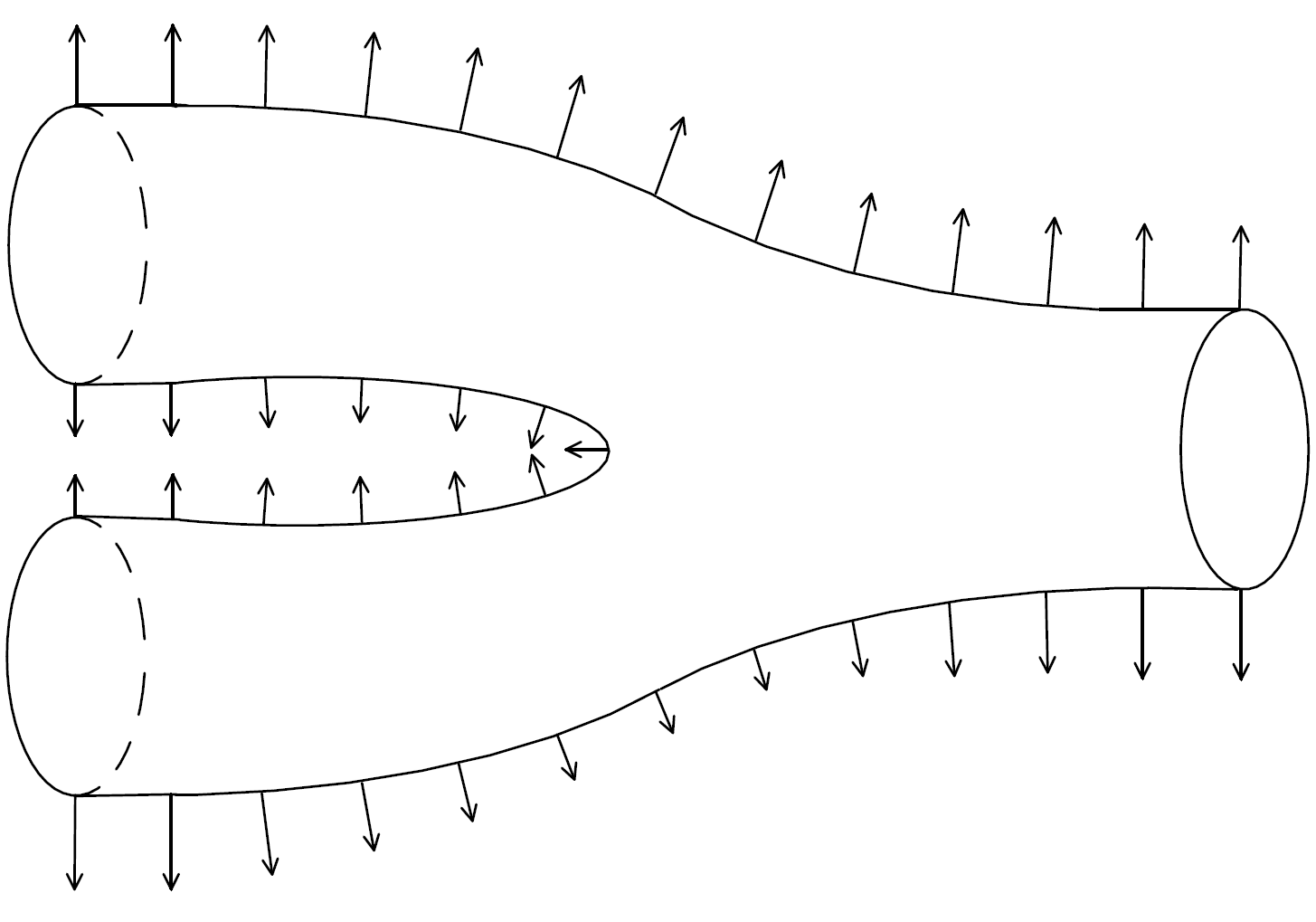}
\caption{\label{framedbordism} Framed cobordism between two 1-dimensional framed submanifolds of $\R^2$}
\end{center}
\squeezeup
\end{figure}
\begin{rem}
$T_xX$ is a ($k$+1)-dimensional and  $T_x(\mathbb{R}^{n+k}\times[0,1])$ is an ($n$+$k$+1)-dimensional vector space for boundary points $x\in \partial X$, too. Thus, $NX$ and framings of $X\subset \mathbb{R}^{n+k}\times[0,1]$ can be defined in the same way as above.
\end{rem}
\begin{rem}
The terms "bordism" and "cobordism" are synonyms. "Bordism" was introduced by Thom, meaning the study of boundaries (from French "bord"). "Cobordism" means that two manifolds jointly form the boundary of a manifold with boundary.
\end{rem}
\noindent It turns out that framed cobordism is an equivalence relation: Reflexi\-vity follows from regarding the cylinder of the underlying framed submanifold together with the constant framing, while symmetry follows from mirroring $\R^{n+k}\times[0,1]$ at the hyperplane $\R^{n+k}\times\{\frac{1}{2}\}$. Finally, transitivity follows from concatenating two framed cobordisms at their boundaries appropriately.\\[1ex] This observation allows us to consider the set $\Omega ^{fr}_k(\mathbb{R}^{n+k})$ of framed cobordism classes of $k$-dimensional framed submanifolds of $\mathbb{R}^{n+k}$. One defines an addition on $\Omega ^{fr}_k(\mathbb{R}^{n+k})$ by mapping two $k$-dimensional framed submanifolds of $\mathbb{R}^{n+k}$ to their union, which can be assumed to be disjoint, since the compactness of the manifolds allows for suitable translations in case they intersect. This operation turns $\Omega ^{fr}_k(\mathbb{R}^{n+k})$ into an abelian group.\\[1ex] Now, consider a smooth map $f:S^{n+k}\to S^n$. Then, if one ignores the tri\-vial case where $f$ is constant, by Sard's Theorem (Theorem \ref{sard}), there is a regular value $y\in S^n$ such that $f^{-1}(y)$ is a closed smooth $k$-submanifold of $S^{n+k}$ by the Preimage Theorem (Theorem \ref{preimagethm}). $f^{-1}(y)$ can be interpreted as a closed smooth $k$-submanifold of $\R^{n+k}$, which can be equipped with a framing. Therefore, one can assign to a smooth map $f:S^{n+k}\to S^n$ a framed cobordism class $[f^{-1}(y)]\in\Omega ^{fr}_k(\mathbb{R}^{n+k})$ in a canonical way. Pontryagin was able to show that two such smooth maps are homotopic if and only if their preimages of a fixed regular value belong to the same framed cobordism class. This observation resulted in the following impressive theorem:
\begin{thm}\label{mthm} There is an isomorphism $\pi_{n+k}(S^n)\cong \Omega ^{fr}_k(\mathbb{R}^{n+k})$.
\end{thm}
\noindent As already mentioned above, we defer the detailed proof of this theorem to the third part and rather keep the preceding ideas in mind to introduce transversality now. 

\section{Transversality}
In the early 1950s, Thom established a more general version of framed cobordism, called cobordism. When doing so, one central question arising was whether for a smooth map $f:M\to N$ between smooth manifolds the preimage $f^{-1}(S)$ of a smooth submani\-fold $S\subset N$ is also a smooth submani\-fold of $M$. When $S$ just consists of a point $y$, as in the case Pontryagin worked on, the Preimage Theorem provides a sufficient condition for this. It can be deduced from the Implicit Function Theorem and thus had already been known for some time.  
\begin{thm}\label{preimagethm}
\textbf{(Preimage Theorem)}
Let $f:M\to N$ be a smooth map between smooth manifolds and $y\in N$ be a regular value of $f$. Then $f^{-1}(y)$ is a smooth submanifold of $M$. If $y\in im(f)$, the codimension of $f^{-1}(y)$ in $M$ is equal to the dimension of $N$ and $T_xf^{-1}(y)=ker(df_x)$ for every $x\in f^{-1}(y)$. 
\end{thm}
\noindent That is, the points $x\in M$ that satisfy $f(x)=y$ for a regular value $y$ of $f$ form a smooth submanifold of $M$.  The Preimage Theorem can be generalized to an arbitrary smooth submanifold $S$ of $N$ instead of a regular value $y$ of $f$, if one requires all points in $S$ to be regular values of f.  However, this condition was too confining for Thom's purposes. Therefore, he searched for less restrictive, but still sufficient conditions to make $f^{-1}(S)$ a smooth submanifold of $M$.\\[1ex]
The problem can be reduced to a local one: $f^{-1}(S)$ is a smooth manifold if and only if for every $x\in f^{-1}(S)$ there is an open neighborhood $U\subset M$ such that $f^{-1}(S)\cap U$ is a smooth manifold. Now let $x\in M$ and $y\in S$  such that $f(x)=y$. Let $n$ denote the codimension of $S$ in $N$. Then, since $S$ is a smooth submanifold of $N$, there exists a neighborhood $V\subset N$ of $y$ and a submersion $\phi: V \to \R ^n$ such that $S\cap V=\phi^{-1}(0)$. Thus, for a neighborhood $U\subset M$ of $x$, $f^{-1}(S)\cap U$ coincides with $(\phi\circ f)^{-1}(0)$. This allows us to make use of a case we already know: By the Preimage Theorem, a sufficient condition for $f^{-1}(S)\cap U$ to be a smooth manifold is that $0$ is a regular value of $\phi\circ f$. In particular, $f^{-1}(S)\cap U$ has codimension $n$ in $M$.\\[1ex]
Now let us investigate in which case this happens. By the chain rule for the differential, we have that $d(\phi\circ f)_x=d\phi _y\circ df_x$. Therefore, $d(\phi\circ f)_x$ is surjective if and only if $d\phi _y$ maps $im(df_x)$ onto $T_{\phi (y)}\R ^n$. However, $d\phi _y$ is a surjective homomorphism with $ker(d\phi _y)=T_yS$. This means that $im(df_x)$ is mapped onto $T_{\phi (y)}\R ^n$ exactly when $im(df_x)$ and $T_yS$ together span the whole space $T_y N$. \\[1ex] 
These considerations motivate the following definition going back to Thom:\newpage
\begin{defn}
Let $f:M\to N$ be a smooth map between smooth manifolds and $S\subset N$ be a smooth submanifold. $f$ is called transverse to $S$ in $x\in M$, if one has the following: \[f(x)\in S \implies T_{f(x)}S + im(df_x)=T_{f(x)}N\] 
$f$ is called transverse to $S$, if $f$ is transverse to $S$ in $x$ for every $x\in M$.
\end{defn}
\begin{rem}
By definition, $f$ is transverse to $\{x\}$ for every $x\notin f^{-1}(S)$ trivially. Authors use the terms "transverse" and "transversal" synonymously.
\end{rem}
\noindent Notice that transversality is a generalization of the concept of regular values: \[f \text{ has regular value } y\in N \Leftrightarrow f \text{ is transverse to }\{y\}\subset N\] 
\begin{exmp}
Let $M=\{x\in\R^3\,|\, x_1=0\}$, $N=\{x\in\R^3 \,|\, x_2=0\}$ and $S=\{x\in\R^3\,|\, x_2=x_3=0\}$ and let \[f:M\to N\text{, }(0,x_2,x_3)\mapsto (0,0,x_3)\] be the projection. Then for every $x\in f^{-1}(S)$ we have that \[T_{f(x)}S + im(df_x)=\{x\in\R^3\,|\, x_2=x_3=0\}+\{x\in\R^3\,|\, x_1=x_2=0\}=T_{f(x)}N.\]
This means that $f$ is transverse to $S$, even though the differential $df_x$ is not surjective for any $x\in M$.
\end{exmp}
\noindent As one can see in the example, if $S$ has dimension $\geq 1$, it is generally less restrictive to require $f$ to be transverse to $S$ instead of requiring every point in $S$ to be a regular value of $f$.
Nevertheless, Thom was able to prove the desired statement while requiring only transversality in his Transversality Theorem.
\begin{thm}\label{trans} \textbf{(Transversality Theorem)}
If a map $f:M\to N$ between smooth manifolds is transverse to a smooth submanifold $S\subset N$, then $f^{-1}(S)$ is a smooth submanifold of $M$. 
If  $f^{-1}(S) \neq \emptyset$, the codimension of $f^{-1}(S)$ in $M$ is equal to the codimension of $S$ in $N$.
\end{thm}
\noindent A proof of this theorem was already given above.
\begin{rem}
When deriving the theorem, we applied the Preimage Theorem to $\phi\circ f$ and its regular value $0$. In particular, we obtain that $T_x((\phi\circ f)^{-1}(0))=ker(d(\phi\circ f)_x)$ for every $x\in (\phi\circ f)^{-1}(0)$. Since $ker(d\phi _{f(x)})=T_{f(x)}S$, we have that $ker(d(\phi\circ f)_x)=(df_x)^{-1}(T_{f(x)}S)$. Therefore, the above equation implies $T_x(f^{-1}(S))=(df_x)^{-1}(T_{f(x)}S)$ for every $x\in f^{-1}(S)$. It follows that 
\[df_x(T_x(f^{-1}(S)))=T_{f(x)}S \text{ for every } x\in f^{-1}(S).\]
\end{rem}
\noindent Such as transversality is a generalization of regularity, the Transversality Theorem can be seen as a generalization of the Preimage Theorem. This is because one can just choose $S=\{y\}$ for a regular value $y$ of $f$.\\[1ex]
\noindent An important special case is the case where the underlying map is an inclusion $i:M\to N$ of a smooth submanifold $M$ of $N$, where $S\subset N$ is another smooth submanifold. Then we have that $i^{-1}(S)=M\cap S$. Moreover, for $x\in M$, $di_x:T_xM\to T_{i(x)}N$ is also just the inclusion  $T_xM\subset T_xN$. Therefore, $i$ is transverse to $S$ if and only if $T_xM+T_xS=T_xN$ for all $x\in M\cap S$. In this case the two smooth submanifolds $M, S \subset N$ are called transverse. \\
\begin{figure}[h]
\squeezeup
\begin{center}
\includegraphics[width = 300pt]{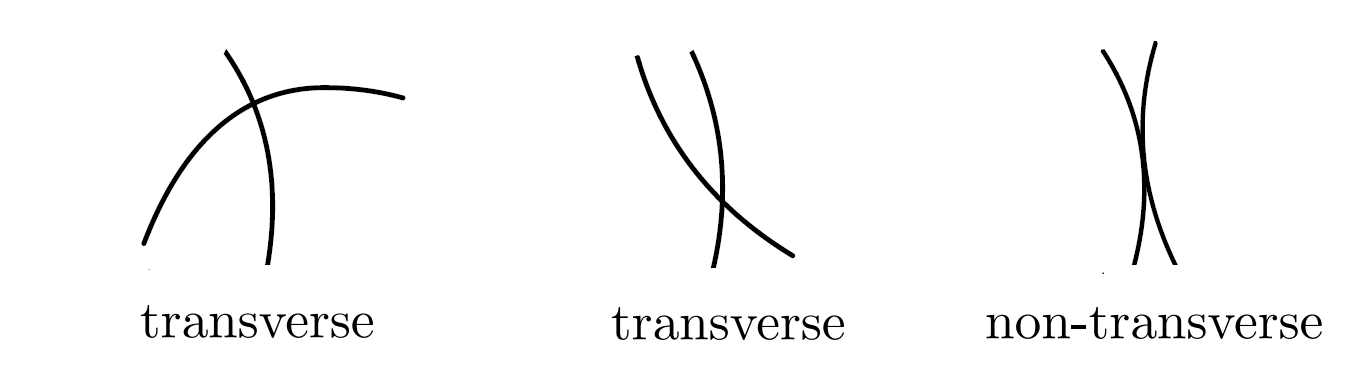}
\caption{\label{intersection} Transverse and non-transverse intersections of two curves in $\R^2$}
\end{center}
\squeezeup
\squeezeup
\end{figure}
\begin{thm}
The intersection of two transverse smooth submanifolds of a smooth manifold $M$ is a smooth submanifold of $M$ again.
\end{thm}
\noindent This follows directly from the Transversality Theorem.
\begin{exmp} Consider the smooth submanifolds $S^2$ and $\R^2\times\{0\}$ of $\R^3$. Let $x\in S^2\cap(\R^2\times\{0\})=S^1\times\{0\}$. Then, under the identification $T_x\R^3\cong \R^3$, we have that $T_xS^2=x^\perp=\{v\in\R^3|v\perp x\}$ and $T_x(\R^2\times\{0\})=\R^2\times\{0\}$. Thus, $T_xS^2+T_x(\R^2\times\{0\})=\R^3\cong T_x\R^3$, so $S^2$ and $\R^2\times\{0\}$ are transverse. The above theorem implies that $S^1\times\{0\}$ is a smooth submanifold of $\R^3$ again.\\
Without the transversality assumption the theorem does not hold anymore: Consider the smooth submanifolds $S^2$ and $\R^2\times\{0,1\}$ of $\R^3$. Let $x\vcentcolon =(0,0,1)\in \R^3$. Then, $S^2\cap(\R^2\times\{0,1\})=(S^1\times\{0\})\cup\{x\}$, which can not be a manifold, since $S^1\times\{0\}$ is $1$-dimensional and $\{x\}$ is $0$-dimensional. Also notice that $T_xS^2+T_x(\R^2\times\{0,1\})=\R^2\times\{0\}+\R^2\times\{0\}\neq T_x\R^3$. Therefore, the theorem can not be applied in this case.
\begin{figure}[h]
\begin{center}
\includegraphics[width = 300pt]{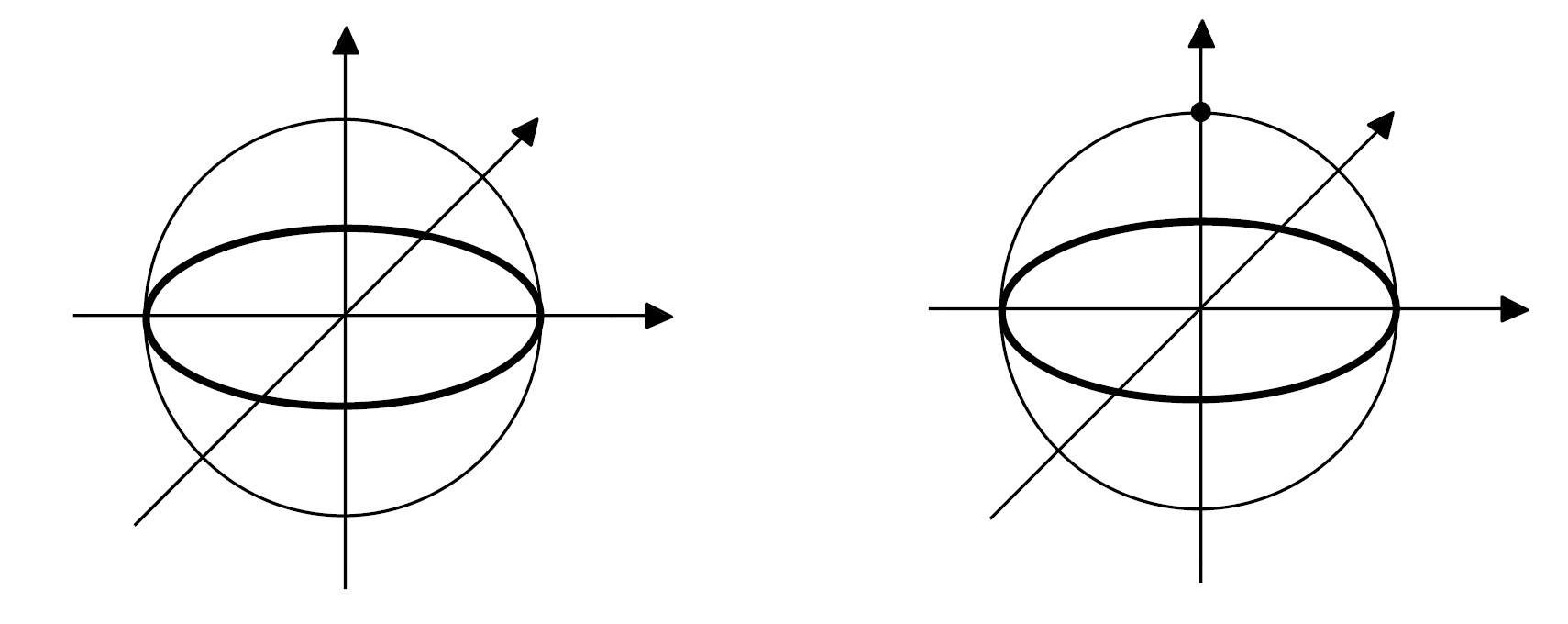}
\caption{\label{example} The smooth submanifold $S^1\times\{0\}\subset \R^3$ and the set $(S^1\times\{0\})\cup\{x\}$, which is not a smooth manifold}
\end{center}
\squeezeup
\squeezeup
\end{figure}
\end{exmp}
\noindent We now focus on another interesting property of transversality discovered by Thom, namely that it is generic. By this we mean that for a smooth submanifold $S$ of $N$ any smooth map $f:M\to N$ can be deformed into a map $g$ transverse to $S$ such that $f$ and $g$ are arbitrarily "close" to each other. Thom gave a proof for this statement in his famous paper from 1954 \cite{Th54}. He regarded the group of diffeomorphisms of $N$ that fix $N\backslash U$ for a tubular neighborhood $U\subset N$ of $S$ and considered the composition of $f$ with such a diffeomorphism that arises from deforming the identity on $N$ by a sufficiently small amount. He showed that this composition then is transverse to $S$. \\ One important corollary of this statement is that every smooth map between smooth manifolds is homotopic to a transverse one. Since Thom's original proof is rather general and extensive, we follow a more elementary approach provided by John Lee (\cite{Le12}, Ch. 6)) and derive genericity of transversality and the homotopy corollary from it.

\begin{thm}\textbf{(Parametric Transversality Theorem)}
Let $M$, $N$, $T$ be smooth manifolds and $S\subset N$ be a smooth submanifold. If a smooth map $F:M\times T\to N$ is transverse to $S$, then for almost every $t\in T$, the map $F_t:M\to N$, $x\mapsto F(x,t)$ is transverse to $S$.
\end{thm}
\begin{rem}
This theorem is just called the Transversality Theorem by some authors. It can be used to derive further transversality statements.
\end{rem}
\begin{proof}
By the Transversality Theorem, $Q\defeq F^{-1}(S)$ is a smooth submani\-fold of $M\times T$. Let $pr:M\times T\to T$ denote the projection onto $T$. We want to show that if $t\in T$ is a regular value of $pr|_Q$, $F_t$ is transverse to $S$. Then the claim follows because almost every $t\in T$ is a regular value of $pr|_Q$ by Sard's Theorem.\\
Let $t\in T$ be a regular value of $pr|_Q$. Choose an $x\in F_t^{-1}(S)$ and define $y\defeq F_t(x)\in S\subset N$. We have to show that 
\[T_yN=T_yS+d(F_t)_x(T_xM).\]
In order to see this equality, we collect what we already know: Since $F$ is transverse to $S$ we have that 
\[T_yN=T_yS+dF_{(x,t)}(T_{(x,t)}(M\times T)).\]
Also, because $t$ is a regular value of $pr|_Q$, it holds that 
\[T_tT=d(pr)_{(x,t)}(T_{(x,t)}Q).\]
Next, by the remark after the Transversality Theorem, we know that 
\[dF_{(x,t)}(T_{(x,t)}Q)=T_yS. \]
In order to show the first equality, for any $v\in T_yN$, there need to exist $u\in T_yS$ and $w\in T_xM$ such that 
\[v=u+d(F_t)_x(w). \]
First, due to the second equality, there are $u_1\in T_yS$ and $(w_1, r_1)\in T_xM\times T_tT\cong T_{(x,t)}(M\times T)$ such that 
\[v=u_1+dF_{(x,t)}(w_1, r_1).\]
Next, the third equation implies that there exists $(w_2,r_2)\in T_{(x,t)}Q$ such that $d(pr)_{(x,t)}(w_2,r_2)=r_1$. As the differential of a projection is the correspon\-ding projection on the tangent spaces, it follows that $r_2=r_1$. Furthermore, notice that
\[dF_{(x,t)}(w_1, r_1)=dF_{(x,t)}(w_2, r_1)+dF_{(x,t)}(w_1-w_2,0).\]
Now the fourth equation implies that $dF_{(x,t)}(w_2, r_1)=dF_{(x,t)}(w_2, r_2)\- \in dF_{(x,t)}(T_{(x,t)}Q)=T_yS$.\\
Let $j_t:M\to M\times T$, $x'\mapsto (x',t)$ denote the inclusion. Then $F_t=F\circ j_t$ and $d(j_t)_x(w_1-w_2)=(w_1-w_2,0)$ which implies $dF_{(x,t)}(w_1-w_2,0)=dF_{(x,t)}\circ d(j_t)_x(w_1-w_2)=d(F_t)_x(w_1-w_2)$. Thus, using the sixth equation, we obtain that the fifth equation holds, if we pick $u=u_1+dF_{(x,t)}(w_2, r_1)$ and $w=w_1-w_2$. This proves the theorem.
\end{proof}
\noindent The Parametric Transversality Theorem allows to explain more precisely what genericity of transversality means. For the intuition consider the special case where $f:M\to\R^k$ is a smooth map into the $k$-dimensional Euclidean space. Denote by $B$ the open ball in $\R^k$ and consider the smooth function \[F:M\times B\to \R^k\text{, }(x,t)\mapsto f(x)+t.\] 
We observe that $dF_{(x,t)}: T_{(x,t)}(M\times B)\to T_{F(x,t)}\R^k$ must be surjective, and therefore, that for any submanifold $S$ of $\R^k$, $F$ is transverse to $S$. Now, the Parametric Transversality Theorem implies that for almost every $t\in B$ the map $F_t:M\to\R^k$, $x\mapsto f(x)+t$ is transverse to $S$. Thus, any smooth map $f:M\to\R^k$ can be changed by an arbitrary small amount into a map which is transverse to $S$.\\[1ex]
\noindent With a bit of extra work, this argument can be generalized to obtain a corresponding statement for smooth maps between arbitrary smooth manifolds. More precisely, Thom showed the following \cite{Th54}:
\begin{thm}\textbf{(Transversality Approximation Theorem)}
Let $f:M\to N$ be a smooth map between smooth manifolds and $S\subset N$ be a smooth submanifold. Let $\delta: M\to \R$ be a positive function and $d$ be a metric defining the topology of $N$. Then there is a smooth map $g:M\to N$ such that $d(f(x),g(x))\leq \delta (x)$ for all $x\in M$ and $g$ is transverse to $Z$. If $f$ is already transverse to $S$ on a closed subset $A\subset N$, $g$ can be chosen such that $f|_A=g|_A$.
\end{thm}
\begin{figure}[h]
\squeezeup
\begin{center}
\includegraphics[width = 300pt]{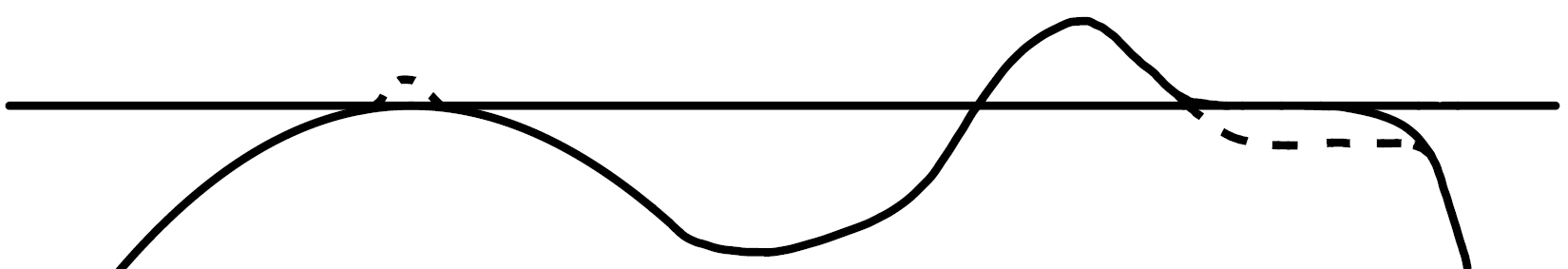}
\caption{\label{approximation} A smooth curve $F:M\to \R^2$ which is not transverse to $\R\times \{0\}\subset \R^2$ is changed by a small amount into a transverse curve}
\end{center}
\squeezeup
\end{figure}

\begin{rem}
Heuristically, this theorem states that the set of smooth maps that are transverse to $S$ is a dense subset of the set of smooth maps. To specify this statement, one needs to choose an appropriate topology on those sets, which, for instance, is done by Morris Hirsch (\cite{Hi76}, Ch. 3.2).
\end{rem}
\noindent The above theorem allows the approximation of smooth maps by transverse ones. Thom used it in combination with an already known approximation theorem going back to Hassler Whitney \cite{Wh36} which allows the approximation of continuous maps by smooth ones in a similar way. One important consequence is that every continuous map between smooth mani\-folds is homotopic to a smooth one (\cite{Le12}, Ch. 6):
\begin{thm} \label{whitney} \textbf{(Whitney Approximation Theorem)}
Let $f:M\to N$ be a continuous map between smooth manifolds. Let $\delta: M\to \R$ be a positive function and $d$ be a metric defining the topology of $N$. Then there is a smooth map $g:M\to N$ such that $d(f(x),g(x))\leq \delta (x)$ for all $x\in M$. If $f$ is already smooth on a closed subset $A\subset N$, $g$ can be chosen such that $f|_A=g|_A$. 
\end{thm}
\begin{cor}\label{whitneycor}
Let $f:M\to N$ be a continuous map between smooth manifolds. Then $f$ is homotopic to a smooth map $g$. If $f$ is already smooth on a closed subset $A\subset M$, $g$ can be chosen such that $f|_A=g|_A$. 
\end{cor}
\noindent There is also a corresponding statement for transverse maps, which can be deduced from the Parametric Transversality Theorem:
\begin{cor} \label{transversalitycor}\textbf{(Transversality Homotopy Theorem)}
Let $f:M\to N$ be a smooth map between smooth manifolds and $S\subset N$ be a smooth submanifold. Then $f$ is homotopic to a smooth map $g$ which is transverse to $S$. If $f$ is already transverse to $S$ on a closed subset $A\subset M$, $g$ can be chosen such that $f|_A=g|_A$. 
\end{cor} 
\begin{proof}
Our goal is to construct a smooth map $F:M\times B\to N$, where $B$ is the open ball in the Euclidean space, such that $F$ is transverse to $S$ and $F_0=f$. Then the Parametric Transversality Theorem implies that $F_t:M\to N$ is transverse to $S$ for some $t\in B$. Now, choosing a path $\gamma :[0,1]\to B$ from $0$ to $t$ allows us to define a homotopy $H(x,s)=F(x,\gamma(s))$ from $F_0=f$ to $F_t$, meaning that $f$ homotopic to the transverse map $F_t$.\\[1ex]
First, Whitney's Embedding Theorem (Theorem \ref{whitneyembedding}) implies that $N$ is a smooth submanifold of $\R^K$ for some $K$. Then, by the Tubular Neighborhood Theorem, there is a tubular neighborhood $U$ of $N$ in $\R^K$ and a diffeomorphism $\psi$ from $U$ to an open neighborhood of the zero section $(NN)_0=\{(y,0)\in NN | y \in N\}$ in the normal bundle $NN$ of $N$. Now let $pr:NN\to N,\text{ } (y,v)\mapsto y$ be the projection. Then $r\defeq pr\circ\psi:U\to N$ is a smooth retraction and also a submersion. Next, define the positive function
\[\delta:N\to \R\text{, } y\mapsto \text{sup}\{\varepsilon\leq 1|B_\varepsilon(y)\subset U\}.\] 
Now Whitney's Approximation Theorem implies the existence of a smooth function $\tilde{\delta}:M\to \R$ such that $0<\tilde{\delta}(x)<\delta (f(x))$ for all $x\in M$. This allows us to define
\[F:M\times B\to N\text{, } (x,t)\mapsto r(f(x)+\tilde{\delta}(x)t).\]
We observe that $F$ is well-defined because $f(x)+\tilde{\delta}(x)t\in U$ due to $|\tilde{\delta}(x)t|< \tilde{\delta}(x)<\delta (f(x))$. Moreover, $F$ is smooth as a composition of smooth maps and $F_0=f$, since $r|_N=id_N$. Finally, for each $x\in M$, $F|_{\{x\}\times B}$ is given by the composition of the local diffeomorphism $t\mapsto f(x)+ \tilde{\delta}(x)t$ and the submersion $r$, so $F$ is a submersion too and, in particular, transverse to $S$.\\[1ex]
If $f$ is already transverse to $S$ on a closed subset $A\subset M$, $\tilde{\delta}$ can be chosen such that $\tilde{\delta}(x)=0$ for $x\in A$ and $\tilde{\delta}(x)>0$ for $x\notin A$ in order to obtain that $f|_A=F_t|_A$, where $F_t$ is transverse to $S$.
\end{proof}
\noindent By combining the last two corollaries, one obtains that for any smooth submanifold $S\subset N$ every continuous map $f:M\to N$ is homotopic to a smooth map $g:M\to N$ which is transverse to $S$. If $f$ is already transverse to $S$ on a closed subset $A\subset M$, then $g$ can be chosen such that $f|_A=g|_A$. This is a powerful statement which played an important role in Thom's work on cobordism theory.\\[1ex]
Before continuing with a concrete application of the above theorems, it should be highlighted that there exist corresponding versions for smooth manifolds with boundary. As an example, we state the boundary version of the Transversality Theorem now. If $f:M\to N$ is a smooth map from a smooth manifold with boundary to a smooth manifold without boundary, denote by $\partial f: \partial M \to N$ the restriction of $f$ to the boundary of $M$.\newpage

\begin{thm}\label{transversalityboundary} \textbf{(Transversality Theorem, Boundary Version)}
Let $f:M\to N$ be a map from a smooth manifold with boundary $M$ to a smooth manifold without boundary $N$ such that $f$ and $\partial f$ are transverse to a smooth submanifold $S\subset N$, then $f^{-1}(S)$ is a smooth submanifold of $M$ with boundary $\partial f^{-1}(S) = f^{-1}(S) \cap \partial M$.
If  $f^{-1}(S) \neq \emptyset$, the codimension of $f^{-1}(S)$ in $M$ is equal to the codimension of $S$ in $N$.
\end{thm}
\noindent We will see an application of the above theorem in the next part. The proof is similar to the proof for the case without boundary, so we omit it here. There are another two useful statements (\cite{Le12}, Ch. 6 and \cite{Gu74}, Ch. 2.3) that can be deduced from the boundary versions of Corollaries \ref{whitneycor} and \ref{transversalitycor} respectively. Together they yield the existence of transverse homotopies between two transverse maps that are homotopic:

\begin{thm}\label{homsmooth}
If $f,g:M\to N$ are homotopic smooth maps between smooth manifolds, then they are smoothly homotopic. If $f$ and $g$ are homotopic relative to some closed subset $A\subset M$, then they are smoothly homotopic relative to $A$.
\end{thm}
\begin{thm}\label{homtrans}
Let $f:M\to N$ be a smooth map from a smooth manifold with boundary to a smooth manifold. Let $A\subset M$ be a closed subset, $S\subset N$ be a closed smooth submanifold. Let $f$ be transverse to $S$ on $A$ and $\partial f:\partial M\to N$ be transverse to $S$ on $A\cap \partial M$. Then there exists a smooth map $g:M\to N$ such that $g$ and $\partial g$ are transverse to $S$ and $f|_A=g|_A$. 
\end{thm}

\noindent Although a variety of further transversality theorems could be stated here, we now continue with a concrete example of the application of transversality. 

\section{An application of transversality}

In this part we mainly focus on proving the one-to-one correspondence bet\-ween homotopy classes of maps from the ($n$+$k$)-sphere to the $n$-sphere and framed cobordism classes of framed $k$-submanifolds of $\R^{n+k}$ described in the first part. This means we apply Thom's concept of transversality to the earlier motivating special case of cobordism, namely framed cobordism, in order to prove a central isomorphism theorem (Theorem \ref{mthm}).\\
The original source for Pontryagin's proof is his work \cite{Po55}. We also use Chapter 7 of John Milnor's little book \cite{Mi65} where one can find a straightforward proof of the one-to-one correspondence we want to show. However, both authors do not work with transversality methods, but with regular values in their proofs. In Pontryagin's case this is obvious, as transversality was not established at that point in time. Since we are interested in applying the above transversality theorems, we give a detailed proof based on a proof sketch by Andrew Putman \cite{Pu15}, who uses transversality in his work.

\noindent Before dedicating our attention to the proof of the isomorphism theorem we need the following preparatory product neighborhood theorem for framed manifolds proved by Milnor \cite{Mi65}:
\begin{thm}\label{prod}
Let ($M,\nu)$ be a $k$-dimensional framed \smf of  $\R^{n+k}$. Then there is an open neighborhood $U$ of $M$ in $\R^{n+k}$ which is diffeomorphic to $M\times\R^n$ such that every $x \in M$ is mapped to $(x,0)\in M\times\R^n$ and each normal frame $\nu(x)$ corresponds to the standard basis of $\R^n$. 
\end{thm}
\begin{proof}
Let us regard the map \[\psi:M\times\R^n\rightarrow\R^{n+k}\text{, \enspace} (x,\eta_1,\dots,\eta_n)\mapsto (x,0)+\eta_1\nu_1(x)+\dots+\eta_n\nu_n(x)  \]
Since $d\psi_{(x,0)}$ is invertible, the inverse function theorem implies that $\psi$ maps an open neighborhood of $(x,0)$ in $M\times\R^n$ diffeomorphically onto an open neighborhood of $x$ in $\R^{n+k}$. Now, in order to generate a contradiction, we assume that $\psi$ is not bijective on the $\varepsilon$-neighborhood $M\times U_{\varepsilon}$ of $M\times\{0\}$ for any $\varepsilon >0$. Then, we find pairs $(x,\eta)\neq(x',\eta')$ in $M\times\R^n$ such that $\|\eta\|,\|\eta'\|>0$ can be chosen arbitrary small and $\psi(x,\eta)=\psi(x',\eta')$.\\
This allows us to find sequences of such pairs with $\eta,\eta'\rightarrow0$. The compactness of $M\times\overline{U_\varepsilon}$ implies that there are convergent subsequences of pairs $(x,\eta)$, $(x',\eta')$ with limits $(x_0,0)$, $(x'_0,0)$. Since they are mapped to the same point under $\psi$, we have that $x_0=x'_0$. This yields a contradiction because $\psi$ would not be one-to-one on any open neighborhood of $(x_0,0)$ then. Thus,
\[\psi:M\times U_{\varepsilon}\rightarrow \psi(M\times U_\varepsilon) \subset\R^{n+k}\] is a diffeomorphism. Notice that \[\varphi: U_{\varepsilon}\to \R^n,\, \eta \mapsto \frac{\eta}{1-\|\eta\|^2\slash \varepsilon^2}\] is a diffeomorphism from $U_{\varepsilon}$ to $\R^n$ with $d\varphi_0=id$, therefore each normal frame $\nu(x)$ still corresponds to the standard basis of $\R^n$. 
\end{proof}
\begin{rem}
Theorem \ref{prod} holds for framed cobordisms $X\subset \R^{n+k}\times[0,1]$ too, i.e. there is an open neighborhood $W\subset \R^{n+k}\times[0,1]$ of $X$ which is diffeomorphic to $X\times \R^n$. The given collars of $X$ allow here for an analogous proof.
\end{rem}
\noindent Let us recall the theorem we would like to prove:
\begin{thm*} \textbf{(Pontryagin)} For all natural numbers $0\leq k <n$ there is an isomorphism \[\pi_{n+k}(S^n)\cong \Omega ^{fr}_k(\mathbb{R}^{n+k}).\]
\end{thm*}

\begin{proof}
We will prove the theorem in 7 steps. In Step 1, we will construct a map $\Phi: \pi_{n+k}(S^n) \rightarrow \Omega ^{fr}_k(\R^{n+k})$, which is well-defined as Steps 2, 3 and 4 will show. Then, Step 5 will yield that $\Phi$ is a homomorphism. Finally, we will obtain the surjectivity and the injectivity of $\Phi$ in Step 6 and 7.\\ 
Using the one-point compactification we will identify $S^m$ with $\mathbb{R}^m\cup \{\infty\}$ in the proof.
\\[1ex]
\textbf{Step 1:} Constructing a map $\Phi: \pi_{n+k}(S^n) \rightarrow \Omega ^{fr}_k(\mathbb{R}^{n+k})$ \\[1ex]
Fix an arbitrary point $y\in \R^n \subset S^n$ and choose a positively oriented basis $\nu=(\nu_1, \dots, \nu_n)$ for the tangent space $T_yS^n$ with the orientation induced by the standard orientation on $\R^{n+1}$. 
 Now let $c\in \pi_{n+k}(S^n)$ be an arbitrary homotopy class. Let $f:\map$ be a pointed representative of $c$ with base points $\infty \in S^{n+k}$ and $\infty \in S^{n}$, i.e. $ f(\infty) = \infty$. 
Using Corollary \ref{whitneycor}, we can assume that $f$ is smooth. Since $f$ is transverse to $y$ in ${\infty}$ trivially, we can suppose that $f$ is transverse to $y$ by Corollary \ref{transversalitycor} meaning that $y$ is a regular value of $f$.\\
The regular value theorem implies that $f^{-1}(y)$ is a $k$-dimensional submani\-fold of $\R^{n+k}$ because $ f(\infty) = \infty$. Furthermore, $ker(df_x)=T_xf^{-1}(y)$ and thus, $df_x$ maps $N_xf^{-1}(y)$ isomorphically onto $T_yS^n$ for every $x\in f^{-1}(y)$. Therefore, we find a unique vector $\omega_i(x) \in N_xf^{-1}(y)\subset T_xS^{n+k}$ that maps to $\nu_i$ under $df_x$ for every $i \in\{1,\dots,n\}$ and every $x\in f^{-1}(y)$. We obtain a pullback framing $(\omega_1, \dots, \omega_n)$ of $f^{-1}(y)$ which we will denote $f^{\ast}\nu\vcentcolon=\omega$. Now we can define \[\Phi(c)\vcentcolon=[(f^{-1}(y),f^{\ast}\nu)] \in \Omega ^{fr}_k(\R^{n+k}).\]
\begin{rem}
If $y\notin im(f)$, $f$ is not surjective and hence null-homotopic. This implies that $c=0 \in \pi_{n+k}(S^n)$, so we need to have $\Phi(c)\vcentcolon=0\in \Omega ^{fr}_k(\R^{n+k})$ in order to make $\Phi$ a homomorphism. However, this matches with our definition because $f^{-1}(y) = \emptyset$ and, moreover, the pullback framing constructed above corresponds to the empty framing $\sigma$, as there is no $x\in f^{-1}(y)$.
\end{rem}
\noindent The framed manifold $(f^{-1}(y),f^{\ast}\nu)$ is called a Pontryagin mani\-fold. In the next steps we need to show that $\Phi$ does not depend on the different choices of the map $f:\map$, the positively oriented basis $\nu$ of $T_yS^n$ and the point $y\in \R^n\subset S^n$ we made during its construction.
\\[1ex]
\textbf{Step 2:} The map $\Phi$ does not depend on the choice of $f$. \\[1ex]
Choose a point $y\in \R^n$ and a positively oriented basis $\nu$ of $T_yS^n$. Now let $f_0, f_1:\map$ be two smooth pointed maps such that they represent the same homotopy class and both are transverse to $y$. Thus, there is a homotopy $F: S^{n+k}\times[0,1]\rightarrow S^n$ from $F_0=f_0$ to $F_1=f_1$ so that $F(\infty,t)=\infty$ for $t\in[0,1]$. Using Theorem \ref{homsmooth}, it can be assumed that $F$ is smooth. We observe that $F$ is transverse to $y$ on the closed set $S^{n+k}\times\{0,1\}\cup\{\infty\}\times[0,1]=\vcentcolon A$ and $F|_{\partial (S^{n+k}\times[0,1])}$ is transverse to $y$ by assumption. By Theorem \ref{homtrans}, there exists a smooth map $G: S^{n+k}\times[0,1]\rightarrow S^n$ such that $G$ is transverse to $y$ and $G|_A=F|_A$. This means that $G$ is a smooth homotopy from $f_0$ to $f_1$ which is transverse to $y$ and satisfies $G(\infty,t)=\infty$ for $t\in[0,1]$. 
We can assume that $G_t=f_0$ and $G_{1-t}=f_1$ for $t\in [0,\varepsilon)$, $\varepsilon >0$.
Now the preimage $G^{-1}(y)\subset\R^{n+k}\times[0,1]$ together with the framing $G^{\ast}\nu$ is a framed cobordism from $f_0^{-1}(y)\subset\R^{n+k}\times\{0\}$ to $f_1^{-1}(y)\subset\R^{n+k}\times\{1\}$, i.e. $\Phi([f_0])=\Phi([f_1])$.
\begin{rem}
Here we used Theorem \ref{transversalityboundary}, i.e. the boundary version of Theorem \ref{trans}, which yields that $G^{-1}(y)$ is a ($k$+1)-dimensional \smf of $\R^{n+k}\times[0,1]$ with boun\-dary $\partial (G^{-1}(y))=G^{-1}(y)\cap \partial M$. This allows us to construct a Pontryagin manifold from $G$ as well, where the framing $G^{\ast}\nu$ of $G^{-1}(y)$ restricted to the boundary coincides with the framings  $f_0^\ast\nu$  of $f_0^{-1}(y)$ and $f_1^\ast\nu$ of $f_1^{-1}(y)$ respectively.
\end{rem}
\noindent \textbf{Step 3:} The map $\Phi$ does not depend on the choice of $\nu$. \\[1ex]
 We can identify the space of all positively oriented bases of $T_yS^n$ with the space $GL^{+}(n,\R)$ of $(n\times n)$-matrices with positive determinant. This is because for every positively oriented base $\upsilon$ there is exactly one transformation matrix $T$ in $GL^{+}(n,\R)$ such that one obtains $\upsilon$ by applying $T$ to a fixed positively oriented base. Now let $\nu=(\nu_1, \dots, \nu_n)$, $\omega=(\omega_1, \dots, \omega_n)$ be two positively oriented bases of $T_yS^n$ for a fixed $y\in \R^n$. Since $GL^{+}(n,\R)$ is path-connected, we can find a smooth path $s:[0,1]\rightarrow GL^{+}(n,\R)$ from $\nu$ to $\omega$. 
We can assume that $s(t)=\nu$ and $s(1-t)=\omega$ for $t\in [0,\varepsilon)$, $\varepsilon >0$. 
Then $f^{-1}(y)\times [0,1]$ together with the framing $\upsilon_i(x,t)=((f^{\ast}(s(t)))_i(x),0)$ is a framed cobordism from $(f^{-1}(y),f^{\ast}\nu)=(f^{-1}(y), f^{\ast}(s(0)))$ to $(f^{-1}(y), f^{\ast}(s(1)))=(f^{-1}(y),f^{\ast}\omega)$. 
\begin{rem}
We will sometimes just write $f^{-1}(y)$ instead of $(f^{-1}(y),f^{\ast}\nu)$ because $[(f^{-1}(y),f^{\ast}\nu)]=[(f^{-1}(y),f^{\ast}\omega)]$ for arbitrary positively oriented bases $\nu, \omega$ of $T_yS^n$.
\end{rem} 
\noindent \textbf{Step 4:} The map $\Phi$ does not depend on the choice of $y$. \\[1ex]
Let $y,z \in \R^n, y\ne z$, be two arbitrary points and $f:\map$ be pointed and transverse to $y$. Our aim is to construct a smooth translation map $\tau: \R^n\rightarrow \R^n$ such that
\begin{enumerate}
\item $\tau(y)=z$;
\item $\tau$ is homotopic to $id_{\R^n}$ through a homotopy which fixes $id_{\R^n}$ on the complement of a compact set, hence $\tau$ differs from $id_{\R^n}$ only on a compact set;
\item the extension $\tau^{\ast}:S^n\rightarrow S^n$ of $\tau$ is smooth and satisfies that $d\tau^\ast_y$ is an orientation-preserving isomorphism;
\item $\tau^{\ast}\circ f$ is transverse to $z$ and $(\tau^{\ast}\circ f)^{-1}(z)=f^{-1}(y)$.
\end{enumerate}
In that case $f$ and $\tau^{\ast}\circ f$ are homotopic maps because the homotopy $\tau \simeq id_{\R^n}$ can be extended to a homotopy $\tau^\ast \simeq id_{S^n}$. They both fix $\infty$ where $f$ is transverse to $y$ and $\tau^{\ast}\circ f$ is transverse to $z$. This means that for a given homotopy class $c\in \pi_{n+k}(S^n)$ in Step 1 $f$ can be chosen to construct $\Phi(c)$ via $y$ and $\tau^{\ast}\circ f$ can be chosen to construct $\Phi(c)$ via $z$. Now Step 3 allows us to choose positively oriented bases $\nu$ of $T_yS^n$ and $\omega$ of $T_zS^n$ such that $d\tau^{\ast}_y$ maps $\nu$ to $\omega$. Then the properties imply that the Pontryagin manifolds obtained by $f$ and $\tau^{\ast}\circ f$ coincide. Thus and because of Step 2, $\Phi$ is independent of the choice of the point $y\in \R^n$ and therefore well-defined. \\[1ex]
Let us construct $\tau$ now: \\[1ex]
Let $\|\text{ }\|$ denote the Euclidean  norm. Define $c\vcentcolon=2max\{\|y\|,\|z\|\}>0$, which implies $y,z\in \overline{B_{\frac{c}{2}}(0)}\subset \R^n$. Consider the smooth cutoff function \[ h:\R \rightarrow \R\text{, } h(r)\vcentcolon = \begin{cases}
1 & \text{for $r\leq 0$}\\
\scalebox{1.25}{$\frac{e^{-\frac{1}{1-r}}}{e^{-\frac{1}{1-r}}+e^{-\frac{1}{r}}}$} & \text{for $0<r<1$}\\
0  & \text{for $r\geq 1$}
\end{cases} 
\]
declining from $1$ to $0$.
Now define the map 
\[\tau:\R^n\rightarrow \R^n \text{, } \tau(x) \vcentcolon=x+h(\frac{\|x\|-2c}{4c})\cdot(z-y),\]
which is a translation of $x$ by $(z-y)$ multiplied with a real number between $0$ and $1$ depending on the distance of $x$ to the origin. $\tau$ is smooth, as it is the composition of smooth functions and we obviously have that $\tau(y)=z$ because $\|y\|\leq \frac{c}{2}<2c$. This implies property 1 of the map $\tau$.
\begin{figure}[h]
\begin{center}
\includegraphics[width = 300pt]{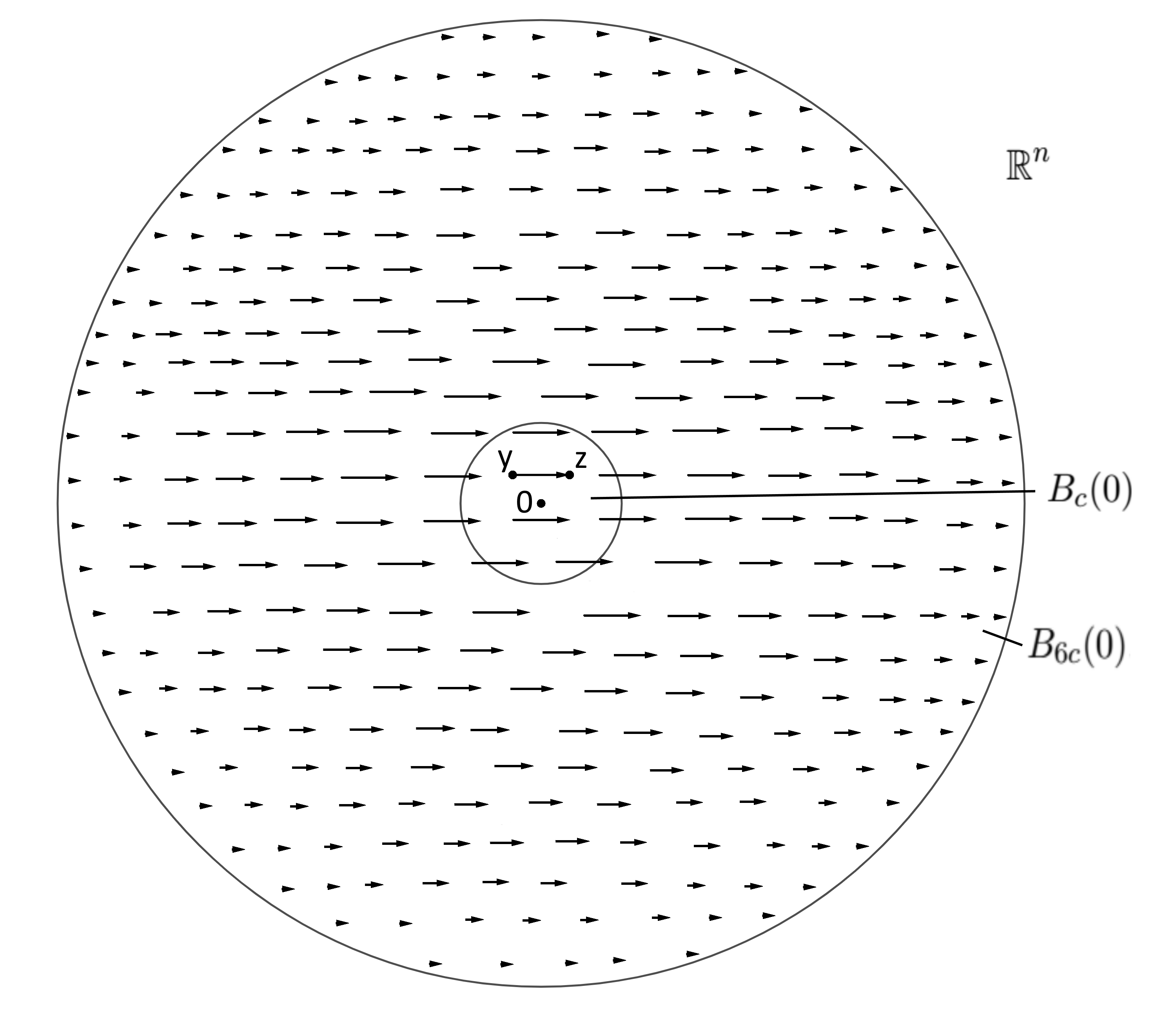}
\caption{\label{tau} The translation map $\tau:\R^n\to\R^n$, which differs from $id_{\R^n}$ only on $\overline{B_{6c}(0)}$}
\end{center}
\squeezeup
\end{figure}
\\Next, we notice that the linear homotopy \[H:\R^n\times [0,1] \rightarrow \R^n \text{, } H(x,t)= x+th(\frac{\|x\|-2c}{4c})\cdot(z-y)\] fixes $id_{\R^n}$ on $\R^n \backslash \overline{B_{6c}(0)}$, which implies property 2.\\
As already mentioned above, we can extend the homotopy $H$ to a homotopy from $id_{S^n}$ to the extension $\tau^{\ast}:S^n\rightarrow S^n$ of $\tau$, where $\tau^{\ast}$ is smooth because it coincides with the identity around $\infty$. We observe that \[\tau^{\ast}|_{B_{c}(0)}:B_{c}(0) \rightarrow B_{c}(z-y)\]
is a diffeomorphism with $(\tau^{\ast})^{-1}(B_{c}(z-y))=B_{c}(0)$ because the only points that can be mapped to $B_{c}(z-y)$ definitely lie in $B_{2c}(0)$ - where $\tau^\ast$ is just the translation by $(z-y)$ - due to the fact that $\|z-y\|\leq c$ and $h\leq 1$. Since translations in $\R^n$ preserve orientation, $d\tau_y$ and hence $d\tau^\ast_y$ are orientation-preserving isomorphisms, which implies property 3.\\
It also follows that \[d(\tau^{\ast}\circ f)_x = d\tau^{\ast}_y\circ df_x: T_xS^{n+k}\rightarrow T_zS^n\] is surjective for every $x\in (\tau^{\ast}\circ f)^{-1}(z)=f^{-1}(y)$ because $f$ is transverse to $y$. This implies property 4.\\[1ex]
Hereby we have shown that $\tau$ meets our requirements.
\begin{rem}
We only need $\tau^\ast$ to be diffeomorphic on $B_{c}(0)$ for our purposes, but one can even show that it is a diffeomorphism from the sphere to the sphere. This is where it is important that the denominator of the fraction put into $h$ is large enough. If one chooses it too small, e.g. $c$ instead of $4c$, one has that $\tau(2(z-y)) = 3(z-y)=\tau(3(z-y))$ for $\|z-y\|=c$ and loses injectivity because $h$ takes the values 1 and 0.
\end{rem}
\noindent Now that we have shown that $\Phi$ is well-defined, we can investigate it further:\\[1ex]
\textbf{Step 5:} The map $\Phi: \pi_{n+k}(S^n) \rightarrow \Omega ^{fr}_k(\R^{n+k})$ is a homomorphism.\\[1ex]
Consider two homotopy classes $c,d \in \pi_{n+k}(S^n)$ represented by pointed smooth maps $f,g:\map$ such that both are transverse to $y\in\R^n$. Now by definition of the composition in the homotopy groups, we obtain $f\ast g:\map$ in the following way: First, we map $S^{n+k}$ to the wedge sum $S^{n+k}\vee S^{n+k}$ by collapsing the equator. Then, we compose it with the map $f\vee g:S^{n+k}\vee \map$ which is defined to be $f$ on the first and $g$ on the second sphere.\\
We observe that $(f\ast g)^{-1}(y)$ is exactly the disjoint union of $f^{-1}(y)$ and $g^{-1}(y)$ because the equator $S^{n+k-1}\subset S^{n+k}$ is mapped to $\infty$. Here we identify the open upper and the open lower hemisphere of $S^{n+k}$ with $\R^{n+k}$, respectively. We obtain the desired result:
\[\Phi(c+d)=\Phi([f\ast g])=[f^{-1}(y)\sqcup g^{-1}(y)]=\Phi([f])+\Phi([g])=\Phi(c)+\Phi(d)\]
\textbf{Step 6:} The map $\Phi$ is surjective.\\[1ex]
Let $(M,\nu)$ be a $k$-dimensional framed \smf of $\R^{n+k}$. Then Theorem \ref{prod} implies that there is an open neighborhood $U\subset \R^{n+k}$ of $M$ together with a diffeomorphism $\psi^{-1}: U\rightarrow M\times \R^n$ such that each normal frame $\nu(x)$ corresponds to the standard basis of $\R^n$. Let $pr: M\times \R^n\rightarrow\R^n$ be the projection and define \[\theta:\map \text{, }\theta(x)\vcentcolon=
\begin{cases} pr\circ \psi^{-1}(x) & \text{for }x\in U\\
\infty & \text{otherwise}\end{cases}\] 
using our common identifications $S^{n+k}=\R^{n+k}\cup \{\infty\}$ and $S^{n}=\R^n\cup \{\infty\}$. Since the collapse map $S^{n+k}\rightarrow S^{n+k}\slash(S^{n+k}\backslash U)$ is continuous, $U^\ast\cong S^{n+k}\slash(S^{n+k}\backslash U)$, $(M\times\R^n)^\ast\cong M\times S^n$ and  $\psi^{-1}$ and $pr$ are proper, we see that $\theta$ is continuous, because the one-point compactification can be understood as a covariant functor from the category of locally compact Hausdorff spaces and proper maps to the category of compact Hausdorff spaces and pointed, particularly continuous, maps. \\
We notice that $\theta|_{U}$ is smooth and transverse to $0\in\R^n\subset S^n$. By Corollary \ref{whitneycor}, there is a smooth approximation $f:\map$ that is homotopic to $\theta$ and satisfies $f(x)\neq 0$ for $x\notin\theta^{-1}(0)$ 
because we can assume that $f|_{A\cup\{\infty\}}=\theta|_{A\cup\{\infty\}}$ for a closed neighborhood $A\subset U$ of $M$. Therefore, $f^{-1}(0)=\theta^{-1}(0)$, $f(\infty)=\infty$ and $f$ is transverse to $0$ . Thus, $f$ can be chosen in Step 1 to construct $\theta$ and, additionally, we have the following:
\[f^{-1}(0)=\theta^{-1}(0)=(\psi \circ pr^{-1})(0)=\psi(M\times\{0\})=M \]
Next, let $b_1,\dots,b_n$ be the basis of $T_0\R^n$ corresponding to the standard basis of $\R^n$ using the identification $T_0\R^n\cong \R^n$.  Now Theorem \ref{prod} implies that $d(pr\circ \psi^{-1})_x$ maps $\nu_1(x),\dots,\nu_n(x)$ to $b_1,\dots,b_n$ for every $x\in M$, since $d(pr)$ can be interpreted as the corresponding projection on the tangent spaces. Under the identification $T_0\R^n \cong T_0S^n$ it follows that \[f^\ast b=\theta^\ast b=(pr\circ \psi^{-1})^\ast b=\nu \text{, \,i.e. }\Phi([f])=[(M,\nu)].\]
\begin{figure}[h]
\begin{center}
\squeezeup
\includegraphics[width = 300pt]{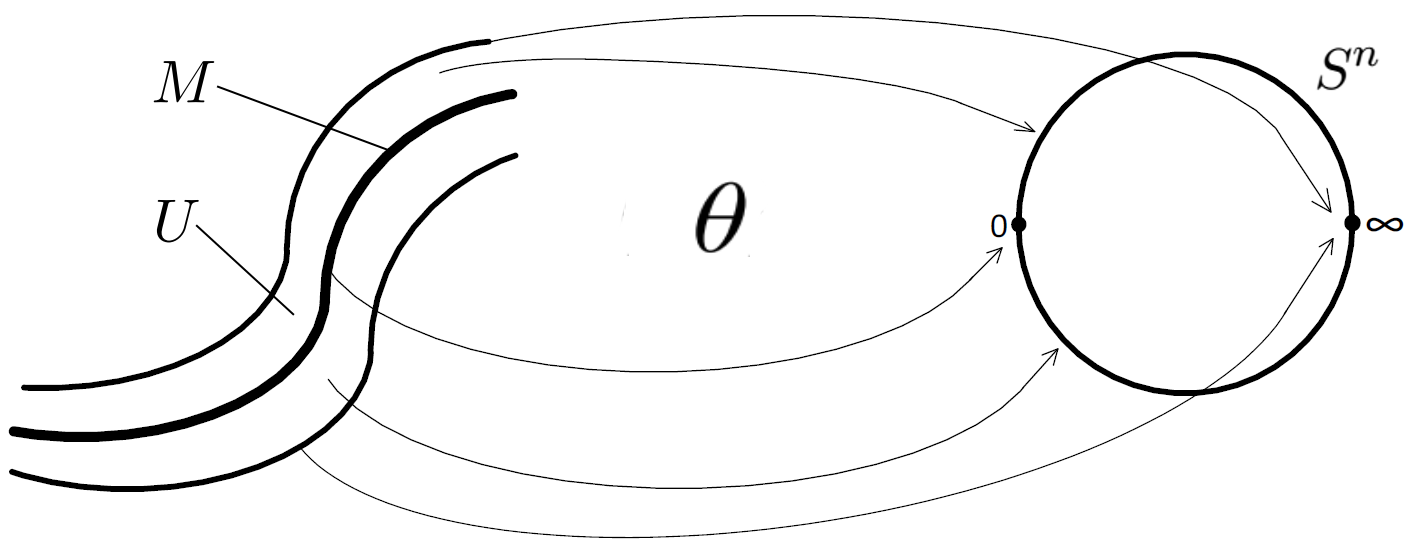}
\caption{ Sketch of the Pontryagin-Thom collapse map $\theta$}
\end{center}
\squeezeup
\end{figure}
\begin{rem}
The map $\theta$ constructed above is called a Pontryagin-Thom collapse map or a Pontrya\-gin-Thom construction. $\theta$ is a continuous map from $S^{n+k}$ to $S^n$ which evaluates the perpendicular distance to $M$ with regard to its framing $\nu$ and $\psi$ such that all points outside a sufficiently small neighborhood $U$ are sent to $\infty$.
\end{rem}
\noindent \textbf{Step 7:} The map $\Phi$ is injective.\\[1ex]
Let $c\in ker(\Phi)$, i.e. $\Phi(c)=0$. Let the smooth pointed map $f_0:\map$ represent $c$ such that $f_0$ is transverse to $0\in\R^n$. Let $b$ be the standard basis of $T_0S^n\cong T_0\R^n\cong \R^n$ and set $(M,\nu)\vcentcolon=(f_0^{-1}(0), f_0^\ast b)$. Now let $f_2:\map$ be the map obtained by applying the Pontryagin-Thom construction from Step 6 to  $(M,\nu)$, which means that $f_2^{-1}(0)=M$ and $f_2^\ast b=\nu$, too. Our aim is to show that $f_0$ and $f_2$ are homotopic maps.\\
At first, we choose a sufficiently small neighborhood $U$ of $M$ in $\R^{n+k}$ such that  we can apply Theorem \ref{prod} and $\infty \notin f_0(U),f_2(U)$. Using the diffeomorphism $\psi: M\times\R^n\cong U$ induced by $\nu$ we can define maps $g_0\vcentcolon=f_0\circ \psi,g_2\vcentcolon=f_2\circ \psi$ from $M\times\R^n$ to $\R^n$, i.e. $g_2$ is the projection as in Step 6. Now let $h:\R\rightarrow\R$ be a cutoff function similar to the one in Step 4 declining from $1$ to $0$ such that $h|_{(-\infty,\varepsilon)} \equiv 1$ for some $\varepsilon>0$. We do not define $h$ explicitly here because it is modified during the proof. Now regard the map
\[H:M\times\R^n\times[0,1]\rightarrow\R^n \text{, } (x,\eta,t)\mapsto g_0(x,\eta)+th(|\eta|)(g_2(x,\eta)-g_0(x,\eta))\]
whose restriction $H|_{V'}$ to a neighborhood $V'\subset U'\vcentcolon=\psi^{-1}(U)$ of $M$ where $h(|\eta|)\equiv 1$ is a homotopy from $g_0|_{V'}$ to $g_2|_{V'}$. Define $g_1\vcentcolon= H_1:M\times\R^n\rightarrow \R^n$ and assume that $h(|\eta|)=0$ for all $(x,\eta)\notin U'$. Then the map
\[f_1:\map\text{, } f_1\vcentcolon=\begin{cases} g_1\circ \psi^{-1} &\text{on }  U \\ f_0 &\text{on } S^{n+k}\backslash U \end{cases}\]
is continuous and homotopic to $f_0$, since $f_1|_{S^{n+k}\backslash U}=f_0|_{S^{n+k}\backslash U}$ and $g_1\simeq g_0 $ via $H$. Furthermore, it satisfies $f_1|_V=f_2|_V$, where $V\vcentcolon=\psi(V')\subset U$.\\
In order to show that $f_1$ is homotopic to $f_2$ we need that $0 \notin f_1(U\backslash V)$. Then we can interpret $f_1|_{S^{n+k}\backslash V}$, $f_2|_{S^{n+k}\backslash V}$ as maps to $\R^n$ by using the identification $S^n = \R^n\cup\{0\}$. Thus, a linear homotopy yields $f_1\simeq f_2$.\\
To show $0 \notin f_1(U\backslash V)$ observe that \[(dg_0)_{(x,0)}=(dg_2)_{(x,0)}=pr_{\R^n}\enspace \forall x\in M\] because $f_0^\ast b=f_2^\ast b$ and the differential of the projection $g_2$ can be interpreted as the corresponding projection on the tangent spaces. Our goal is to find a constant $a>0$ such that \[\langle g_0(x,\eta), \eta \rangle >0 \text{ and } \langle g_2(x,\eta), \eta\rangle >0\]
for all $x\in M$ and $\eta \in \R^n$ with $0<\|\eta\|<a$. This means that $g_0(x,\eta)$ and $g_2(x,\eta)$ belong to the same open half-space in $\R^n$, which implies that $g_1(x,\eta)\neq 0$ because it lies on the line between them. \\
Since $(dg_0)_{(x,0)}=pr_{\R^n} \enspace \forall x\in M$, Taylor’s Theorem implies
\begin{flushleft}
$g_0(x,\eta)=g_0(x,0)+(dg_0)_{(x,0)}((x,\eta)-(x,0))+a(x,\eta)\|(x,\eta)-(x,0)\|^2$\\[1ex]
$=0+\eta+a(x,\eta)\|\eta\|^2\text{ for a smooth }a:M\times\R^n\rightarrow \R^n$\\[1ex]
Now define $a_0\vcentcolon=max\{\|a(x,\eta)\|:\|\eta\|\leq 1\}$ \\[1ex]
$\implies\|g_0(x,\eta)-\eta\|\leq a_0\|\eta\|^2$ for $\|\eta\|\leq 1$\\[1ex]
$\implies|\langle g_0(x,\eta)-\eta,\eta\rangle|\leq \|g_0(x,\eta)-\eta\|\|\eta\|\leq a_0\|\eta\|^3$\\[1ex]
$\implies \langle g_0(x,\eta),\eta\rangle \geq \|\eta\|^2-a_0\|\eta\|^3>0 \text{ for }0<\|\eta\|<min\{a_0^{-1},1\}$
\end{flushleft}
Analogously, one obtains such an inequality and an $a_2>0$ for $g_2$. Now $a\vcentcolon=min\{a_0^{-1},a_2^{-1},1\}$ is our desired constant.\\
Now we can adjust the cutoff function $h$ such that $h(r)=0 \enspace \forall r\geq a$ making it decline to $0$ very quickly ($V$ and $V'$ may be getting smaller). It follows that $0 \notin f_1(U\backslash V)$ and hence, $f_0 \simeq f_1 \simeq f_2$.\\
Next, consider the framed cobordism $X\subset\R^{n+k}\times[0,1]$ from $M$ to the empty manifold $\emptyset$, which exists by assumption, together with a neighborhood $W$ of $X$ in $\R^{n+k}\times[0,1]\subset\R^{n+k+1}$ as in Theorem \ref{prod} that is diffeomorphic to $X\times\R^n$. We can assume that $W\cap(\R^{n+k}\times\{0\})=U$ (otherwise $U$ needs to be chosen smaller). Then we can apply the Pontryagin-Thom construction from Step 6 to $X$ to obtain a homotopy
$G: S^{n+k}\times [0,1] \rightarrow S^n $
with $G_0=f_2$ and $G_1\equiv\infty$, as $W$ can be chosen such that its intersection with $\R^{n+k}\times\{1\}$ is empty. Finally, it follows that the original representative $f_0$ of the homotopy class $c\in ker(\Phi)$ is homotopic to the constant map $G_1$, i.e. $c=0$.
\end{proof}

\noindent
The benefit of using transversality in this proof is that one can choose an arbitrary point $y\in \R^n\subset S^n$ in Step 1, since the representative $f$ of a homotopy class $c\in \pi_{n+k}(S^n)$ can be assumed to be transverse to $y$ by the approximation statements \ref{whitneycor} and \ref{transversalitycor}. This facilitates defining $\Phi: \pi_{n+k}(S^n) \rightarrow \Omega ^{fr}_k(\mathbb{R}^{n+k})$ explicitly and therefore also simplifies the argument in Step 5 that $\Phi$ is a homomorphism. When using an approach without transversality, one needs to choose another point $y'\in \R^n\subset S^n$ close to $y$ which is a regular value of $f$ (exists by Sard's Theorem). Then, one needs to show that the framed cobordism class of the corresponding Pontryagin manifold does not depend on the choice of $y'$, which requires a technical argument. \\[1ex]
Above, a framed cobordism theorem was proved using transverality. Never\-theless, as already mentioned in the introduction, the main purpose of  transversality was cobordism theory for Thom. He used it to generalize Pontryagin's work on framed cobordism in a remarkable way.  Like other ideas of Thom transversality has many more applications and is still signi\-ficant today.

\newpage
\appendix 
\section{Appendix}
\begin{thm} \label{sard} \textbf{(Sard's Theorem)}
Let $f:M\to N$ be a smooth map between smooth manifolds. Then almost every point in $N$ is a regular value of $f$, i.e. the set of critical values of $f$ has Lebesgue measure zero. \cite{Sa42}
\end{thm}

\begin{thm} \label{whitneyembedding} \textbf{(Whitney's Embedding Theorem)}
Every smooth mani\-fold of positive dimension $K$ can be embedded smoothly into $\R^{2K}$. \cite{Wh44}

\end{thm}

\section{References}
\begingroup
\renewcommand{\section}[2]{}%

\endgroup
\end{document}